\documentclass[11pt]{amsart}
\usepackage{amsmath, amsthm,amsxtra}
\usepackage[english]{babel}
\usepackage[T1]{fontenc}
\usepackage[latin1]{inputenc}
\usepackage{amssymb}
\usepackage{amscd}
\usepackage{latexsym}
\usepackage[bookmarksnumbered,plainpages,hypertex]{hyperref}
\usepackage{graphicx}
\usepackage{mathrsfs} 
\usepackage[all,cmtip]{xy}

\newtheoremstyle{theorem}
  {12pt}          
  {12pt}  
  {\sl}  
  {\parindent}     
  {\bf}  
  {. }    
  { }    
  {}     
\theoremstyle{theorem}
\newtheorem{theorem}{Theorem}
\newtheorem{corollary}[theorem]{Corollary}
\newtheorem{remark}[theorem]{Remark}
\newtheorem{proposition}[theorem]{Proposition}
\newtheorem{lemma}[theorem]{Lemma}

\newtheorem{definition}[theorem]{Definition}

\linethickness{0.5pt}

\newcommand{\ic}{\ensuremath{\mathcal{I}}}
\newcommand{\gc}{\ensuremath{\mathcal{G}}}
\newcommand{\oc}{\ensuremath{\mathcal{O}}}
\newcommand{\ac}{\ensuremath{\mathcal{A}}}
\newcommand{\fc}{\ensuremath{\mathcal{F}}}

\newcommand{\ec}{\ensuremath{\mathcal{E}}}

\newcommand{\lc}{\ensuremath{\mathcal{L}}}

\newcommand{\nc}{\ensuremath{\mathcal{N}}}

\newcommand{\jc}{\ensuremath{\mathcal{J}}}

\newcommand{\Pt}{\mathbb{P}^3}

\newcommand{\Pun}{\mathbb{P}^1}
\newcommand{\Pn}{\mathbb{P}^n}
\newcommand{\Pk}{\mathbb{P}^k}

\newcommand{\bZ}{\mathbb{Z}}

\newcommand{\bC}{\mathbb{C}}

\newcommand{\aG}{\alpha}
\newcommand{\bG}{\beta}

\newcommand{\rG}{\rho}
\newcommand{\oG}{\omega}
\newcommand{\lG}{\lambda}

\newcommand{\fG}{\varphi}

\newcommand{\lag}{\langle}
\newcommand{\rag}{\rangle}

\newcommand{\bds}{\begin{displaystyle}}
\newcommand{\eds}{\end{displaystyle}}

\title[Buchsbaum index of rank two vector bundles.]{On the Buchsbaum index of rank two vector bundles on $\Pt$.}

\author{Philippe Ellia - Laurent Gruson}
\address{Dipartimento di Matematica, 35 via Machiavelli, 44100 Ferrara}
\email{phe@unife.it}

\address{L.M.V., Université de Versailles-St Quentin, 45 Av. des Etats-Unis, 78035 Versailles, France}
\email{Laurent.Gruson@uvsq.fr}


\begin{document}
\maketitle


\thispagestyle{empty}

{\it Dedicated to Emilia Mezzetti on her sixtieth birthday.}

\section{Introduction.} We work over an algebraically closed field of characteristic zero. Let $E$ be a rank two vector bundle on $\Pt$. The Buchsbaum index of $E$ is $b(E):= min\,\{k\mid \frak m^k.H^1_*(E)=0\}$ (in the literature one often says that $E$ is ''$k$-Buchsbaum''). By Horrock's theorem $b(E)=0$ if and only if $E$ is the direct sum of two line bundles. Then we have (see \cite{Ellia-Sarti}):

\begin{theorem}
\label{T-b<3}
Let $E$ be a normalized rank two vector bundle on $\Pt$.\\
(1) If $b(E)=1$, then $E$ is a null-correlation bundle.\\
(2) If $b(E)=2$, then $E$ is stable with $c_1=0, c_2=2$ (an instanton with $c_2=2$).
\end{theorem}

This classification is quite simple. However since every bundle is $k$-Buchsbaum for some $k$ it is clear that soon or later we will reach a point where the classification will be intractable. Since there were some echoes on Buchsbaum bundles during the conference we were curious to see if it was possible to push the classification a little bit further. Our result is as follows:

\begin{theorem}
\label{T-the thm}
let $E$ be a rank two vector bundle on $\Pt$ with $b(E)=3$. Then $E$ is stable and:\\
(i) if $c_1(E)=0$, $E$ is an instanton with $3 \leq c_2(E) \leq 5$. Moreover for any $3\leq c_2 \leq 5$, there exists an instanton, $E$, with $c_2(E)=c_2$ and $b(E)=3$.\\
(ii) if $c_1(E)=-1$, then $c_2=2$. Every stable bundle $E$ with $c_1=-1, c_2=2$ has $b(E)=3$.
\end{theorem}

This answers a conjecture made in \cite{JM}. The main tools we use are a restriction theorem (Theorem \ref{T-res}) to control $h^0(E_H(1))$ in the stable case, some general properties (see Proposition \ref{P-generalities b(E)}) and a careful study of the minimal monad of Horrocks build from the minimal free resolution of the module $H^1_*(E)$.

In particular in Section \ref{sec-neg inst} we investigate the minimal monad of ''negative instantons'' (a negative instanton is a stable bundle $E$ with $c_1(E)=-1$ and $h^1(E(-2))=0$). Contrary to what happens in the case of ''positive'' instanton ($c_1=0$) the $H^1_*$ module is not necessarily generated by its elements of degree -1, some generators of degree zero may occur. If $c$ denotes the number of generators of degree zero, it is easy to show that $c \leq c_2/2$. In fact in a forthcoming paper (\cite{EG-2}) we prove:

\begin{theorem}
\label{T-c<n}
Let $E$ be a negative instanton with $c_2 \geq 2$. Then $c \leq c_2/2 -1$. Moreover if $c=c_2/2-1$, then $h^0(E(1)) =1$. Finally for every $c_2 \geq 2$ there exists a negative instanton with $c = c_2/2 -1$.
\end{theorem}

However to prove Theorem \ref{T-the thm}, we need this result just for $c_2 \leq 6$. So to keep this paper self-contained we will prove this particular case with an ad-hoc argument (see Proposition \ref{P-bd c n<4}, Corollary \ref{C-pf Prop}).

To conclude let us make this curious remark: every vector bundle $E$, with $1 \leq b(E) \leq 3$ is an instanton (positive or negative).


\section{Generalities.}

A first bound on the Buchsbaum index of $E$ is given by the diameter of $H^1_*(E)$:

\begin{definition}
\label{D-diameter}
The diameter of the indecomposable rank two vector bundle $E$ is $d(E):= c-c'+1$, where $c = max\{k \mid h^1(E(k)) \neq 0\}$, $c'=min\{k \mid h^1(E(k)) \neq 0\}$.
\end{definition}

We have (see \cite{Buraggina}):

\begin{theorem}
\label{T-Buraggina}
Let $E$ be a rank two vector bundle on $\Pt$, then $H^1_*(E)$ is connected (i.e. if $h^1(E(k))=0$ for some $k > c'$, then $h^1(E(m))=0$ for $m \geq k$).
\end{theorem}

It follows that the diameter counts the number of non-zero (successive) pieces in the module $H^1_*(E)$ and that $b(E) \leq d(E)$.
\medskip

The following result, which may be considered as a complement to Barth's restriction theorem, will play an important role:

\begin{theorem}
\label{T-res}
Let $E$ be a stable, normalized, rank two vector on $\Pt$ with $c_2 \geq 4$. If $H$ is a general plane then: $h^0(E_H(1)) \leq 2+c_1$. In particular $h^0(E(1)) \leq 2+c_1$.
\end{theorem}

\begin{proof} See \cite{EG}.
\end{proof}

Here we collect some general properties:

\begin{proposition}
\label{P-generalities b(E)}
Let $E$ be a normalized, rank two vector bundle on $\Pt$.
\begin{enumerate}
\item Assume all the minimal generators of $H^1_*(E)$ are concentrated in one and the same degree (i.e. $H^1(E(c'))$ generates $H^1_*(E)$)). Then $d(E)=b(E)$.\\
\item Let $\aG$ be the greatest degree of a minimal generator of $H^1_*(E)$. Then $h^1(E(n))=0$, if $n \geq \aG+b(E)$.\\
\item If $E$ is stable, then $h^1(E(-k))=0$ for $k \geq b(E)$. Moreover $h^1(E(-b+1)) \leq h^0(E_H(1))$; if $c_2 \geq 4$ then: $h^1(E(-b+1))\leq h^0(E_H(1)) \leq 2+c_1$ ($H$ a general plane).\\
\item If $E$ is stable with $c_1=-1, c_2\geq 4$ and if $b(E) \geq 3$, then $h^1(E(-b+1))=0$.
\end{enumerate} 
\end{proposition}

\begin{proof} (1) Assume $H^1(E(k))$ generates $H^1_*(E)$, then $c'=k$. The natural map $H^1(E(k))\otimes S^{c-k}(V) \to H^1(E(c))$ is surjective and non-zero. It follows that $b(E)=c-c'+1=d(E)$.\\
(2) We have the minimal free resolution: $...\to \bigoplus  S(-a_i) \oplus k .S(-\aG ) \to H^1_*(E) \to 0$, where $a_i < \aG$. Twisting by $\aG +b$ and using the fact that $\frak m^b.\xi =0$ for any generator $\xi$, we get $H^1(E(\aG +b))=0$. We conclude with Theorem \ref{T-Buraggina}.\\
(3) It is enough to show $h^1(E(-b))=0$ ($b=b(E)$). Since $h^0(E_H)=0$ by Barth's theorem if $H$ is a general plane, we have an injection $H^1(E(-b)) \stackrel{.H^b}{\hookrightarrow} H^1(E)$. Since $.H^b=0$, $h^1(E(-b))=0$.

In the same way we have an injection $H^1(E(-b+1))\stackrel{.H^{b-1}}{\hookrightarrow} H^1(E)$. Composing with $H^1(E) \stackrel{.H}{\to} H^1(E(1))$ we must get zero, so the image of $H^1(E(-b+1))$ in $H^1(E)$ is contained in the kernel $K_H$ of $H^1(E) \stackrel{.H}{\to} H^1(E(1))$. Since $H^0(E_H(1))$ surjects $K_H$ we get $h^0(E_H(1)) \geq h^1(E(-b+1))$. We conclude with Theorem \ref{T-res}.\\
(4) If $h^1(E(-b+1))\neq 0$, by (3) $h^1(E(-b+1))=1$. Let $L$ be a general line. By combining $0 \to \ic _L \to \oc \to \oc _L\to 0$ and $0 \to \oc (-2) \to 2.\oc (-1) \to \ic _L \to 0$ twisted by $E(-b+2)$ we get:
$$\begin{array}{ccccccccc}
 & & & & &  &0 & & \\
 & & & & &  &\downarrow & & \\
0 & \to & E(-b) & \to& 2.E(-b+1) & \to & \ic _L\otimes E(-b+2)& \to & 0\\
 & & & & &  &\downarrow & & \\
 & & & & &  &E(-b+2) & & \\
 & & & & &  &\downarrow & & \\
 & & & & &  &E_L(-b+2) & & \\
 & & & & &  &\downarrow & & \\
 & & & & &  &0 & & \end{array}$$
Taking cohomology since $h^1(E(-b))=0$ and $h^0(E_L(-b+2))=0$ (indeed $E_L(-b+2) \simeq \oc _L(-b+2)\oplus \oc _L(-b+1)$), we get $2.H^1(E(-b+1)) \hookrightarrow H^1(E(-b+2))$. It follows that the map $H^1(E(-b+1))\otimes V \to H^1(E(-b+2))$ has an image, $W$, of dimension at least two. Now we have an injective map $H^1(E(-b+2)) \stackrel{.H^{b-2}}{\hookrightarrow} H^1(E)$. So $W':= .H^{b-2}(W) \subset H^1(E)$ has dimension at least two. Since $W'$ has to be contained in the kernel, $K_H$, of $H^1(E) \stackrel{.H}{\to}H^1(E(1))$ and since $h^0(E_H(1)) \geq \dim (K_H)$, we get a contradiction (see (3)).      
\end{proof}

We recall the following fact (see for instance \cite{Ha-Rao} Prop. 3.1, this is stated for $c_1=0$ but works also for $c_1=-1$):

\begin{lemma}
\label{L-spe et gene H^1}
Let $E$ be a stable, normalized, rank two vector bundle on $\Pt$. Let $\{k_i\}$ be its spectrum. Set $k_+=max\,\{k_i\}$. Then $H^1_*(E)$ is generated in degrees $\leq k_+-c_1-1$.\\
Let $\rho (k)$ denote the number of minimal generators of $H^1_*(E)$ in degree $k$, then: $\rho (-1-j)\leq s(j)-1$, for $0\leq j \leq k_+$ (here $s(j) =\#\,\{j\mid k_i=j\}$.
\end{lemma}

Finally let us recall Horrock's construction of the ''minimal monad'' for a rank two vector bundle $E$ on $\Pt$ with $-1\leq c_1\leq 0$. Let
$$\cdots \to L_2 \to L_1 \to L_0 \to H^1_*(E) \to 0$$
be the minimal free resolution. Then $L_1 \simeq L_1^*(c_1)$, $L_2$ has a direct summand isomorphic to $L_0^*(c_1)$ which induces a minimal monad
$$\tilde L_0^*(c_1) \hookrightarrow \tilde L_1 \twoheadrightarrow \tilde L_0$$
whose cohomology is $E$. Furthermore $rk(L_1) = 2rk(L_0)+2$. See for instance \cite{Horrocks}, \cite{Decker}, \cite{Ha-Rao}. 

\section{Unstable bundles.}

First of all let us recall the following useful fact:

\begin{lemma}
\label{L-res Strano}
Let $\ec$ be a rank two vector bundle on $\Pt$ with $c_1(\ec )=c_1$. Assume $\ec$ has a section vanishing in codimension two. If $h^0(\ec _H(-c_1+1)) \neq 0$ for $H$ a general plane, then $\ec$ is the direct sum of two line bundles.
\end{lemma}

\begin{proof} We have an exact sequence: $0 \to \oc \to \ec \to \ic _X(c_1)\to 0$, where $X \subset \Pt$ is a curve. The assumption implies $h^0(\ic _{X\cap H}(1))\neq 0$, for $H$ a general plane. By a result of Strano (\cite{Strano}) this implies ($ch(k)=0$) that $X$ is a plane curve. So $H^1_*(\ec )=0$ and $\ec$ is decomposed.
\end{proof}

\begin{remark} The assumption $ch(k)=0$ is necessary, see \cite{Ha-Fe}.
\end{remark}

From now on $E$ will denote a normalized, unstable rank two vector bundle. hence $h^0(E(-r))\neq 0$ for some $r \geq 0$ and we will assume that $-r$ is the least twist having a section.

\begin{lemma}
\label{L-h^1=0 unstable}
With notations as above, if $b(E)=b$, then $h^1(E(r-b+1-c_1))=0$. Moreover: $-2b+2r+6-c_1 \leq 0$.
\end{lemma}

\begin{proof} By assumption we have:
$$0 \to \oc \to E(-r) \to \ic _C(-2r+c_1) \to 0$$
where $C$ is a curve with $\oG _C(4+2r-c_1)\simeq \oc _C$. In particular: $1-p_a(C) = d(4+2r-c_1)/2\,\,(*)$, where $d=\deg (C)$.

We may assume $h^0(\ic _C(1))=0$ (otherwise $E$ is decomposed). It follows that $h^0(E(k))= h^0(\oc (k+r))$ if $k \leq r-c_1+1$. We may assume $h^0(\ic _{C\cap H}(1))=0$ if $H$ is a general plane (Lemma \ref{L-res Strano}). Hence $h^0(E_H(k))=h^0(\oc _H(k+r))$ if $k \leq r-c_1+1$. This shows that:
$$0 \to E(k-1) \to E(k) \to E_H(k) \to 0$$
induces an exact sequence on global sections if $k \leq r-c_1+1$. So $H^1(E(k-1))\stackrel{.H}{\hookrightarrow} H^1(E(k))$ if $k \leq r-c_1+1$. Then $H^1(E(r-b-c_1+1)) \stackrel{.H^b}{\hookrightarrow} H^1(E(r+1-c_1))$ is injective. Since $.H^b\equiv 0$, $h^1(E(r-b-c_1+1))=0$.

It follows that $h^1(\ic _C(-b+1))=0 = h^0(\oc _C(-b+1))$. This implies $\chi (\oc _C(-b+1)) \leq 0$. Since $\chi (\oc _C(-b+1))=d(-b+1) -p_a(C)+1$, from $(*)$ we get: $-2b+2r-c_1+6 \leq 0$.
\end{proof}

This gives us the complete classification when $b \leq 3$:

\begin{proposition}
\label{P-unstable case b=3}
There is no unstable rank two vector bundle $E$ with $1\leq b(E) \leq 3$.
\end{proposition}

\begin{proof} From $-2b+2r+6-c_1 \leq 0$ (Lemma \ref{L-h^1=0 unstable}), since $r \geq 0$ and $-1\leq c_1\leq 0$, we see that if $b \leq 3$ the only possibility is $b=3$, $c_1=r=0$. So $E$ is properly semi-stable, with $h^1(E(-2))=0$ (Lemma \ref{L-h^1=0 unstable}). By Serre's duality we have $h^2(E(k))=0$ if $k\geq -2$. Also $E_L \simeq 2.\oc _L$ for a general line. Combining the exact sequences: $0 \to E(m)\otimes \ic _L \to E(m) \to E_L(m) \to 0$ and $0 \to E(m-2) \to 2.E(m-1) \to E(m)\otimes \ic _L \to 0$, we see that $2.H^1(E(m-1)) \to H^1(E(m))$ is surjective. Hence $H^1_*(E)$ is generated by $H^1(E(-1))$. It follows (Proposition \ref{P-generalities b(E)}) that $d(E)=b(E)$. If $b(E)=3$, then $h^1(E(2))=0$. Finally we get $\chi (E(2)) = h^0(E(2)) = 20-4c_2$. Since $h^0(E(2)) \geq h^0(\oc (2))=10$, we get $c_2 \leq 2$. So the section of $E$ vanishes along a curve of degree two with $\oG _C(4)=\oc _C$. So $C$ is a double line of arithmetic genus -3. But the Hartshorne-Rao module of such a curve has diameter 5.

\end{proof}

\section{Stable bundles with $c_1=0$ and $b=3$.}

Let $E$ be a stable bundle with $c_1=0$ and $b(E)=3$. By Proposition \ref{P-generalities b(E)}, $h^1(E(-3))=0$ and $h^1(E(-2)) \leq 2$. We will distinguish two cases: (a) $h^1(E(-2)) > 0$, (b) $h^1(E(-2))=0$.

\subsection{Stable bundles with $c_1=0, b=3$ and $h^1(E(-2)) > 0$.}\quad \\

We first observe that by the properties of the spectrum (\cite{SRS}, 7.1, 7.2, 7.5) the spectrum of $E$  is of the form $\{-1^u, 0^{c_2-2u}, 1^u\}$, where $u := h^1(E(-2))\leq 2$. It follows (Lemma \ref{L-spe et gene H^1}) that $H^1_*(E)$ is generated in degrees $\leq 0$ and $h^1(E(3))=0$ (Proposition \ref{P-generalities b(E)}). 

\begin{lemma}
\label{L-bd c2 c1=0}
Let $E$ be a stable rank two vector bundle with $c_1=0$ and $b(E)=3$. Then $c_2(E) \leq 8$ and if $H^1_*(E)$ is generated in degrees $\leq -1$, $c_2(E)\leq 5$.
\end{lemma}

\begin{proof} As already said $h^1(E(3))=0$. This implies $\chi (E(3)) = 40 -5c_2 \geq 0$, hence $c_2 \leq 8$. If $H^1_*(E)$ is generated in degrees $\leq -1$, by Proposition \ref{P-generalities b(E)}, we have $h^1(E(2))=0$, this implies $\chi (E(2)) = 20-4c_2 \geq 0$, hence $c_2 \leq 5$.
\end{proof}

The following is well known (\cite{Ellia-vanish}) but for the convenience of the reader we include a proof:

\begin{lemma}
\label{L-minMonad c1=0 -1}
Let $E$ be a stable rank two vector bundle with $c_1=0$ and spectrum $\{-1, 0^{c_2-2},1\}$. Then $H^1_*(E)$ is generated in degrees $\leq -1$. More precisely:\\
(i) If the natural map $\mu :H^1(E(-2))\otimes V \to H^1(E(-1))$ is injective the minimal monad has the following shape:
$$(c_2-4).\oc (-1)\oplus \oc (-2) \hookrightarrow (2c_2-4).\oc \to (c_2-4).\oc (1)\oplus \oc (2)$$
(ii) If $\mu$ is not injective it has rank three and the minimal monad is:
$$(c_2-3).\oc (-1)\oplus \oc (-2) \hookrightarrow \oc (-1)\oplus (2c_2-4).\oc \oplus \oc (1) \to (c_2-3).\oc (1)\oplus \oc (2)$$
\end{lemma}

\begin{proof} (i) We know that $M:=H^1_*(E)$ is generated in degrees $\leq 0$. If $\mu$ is injective there are $c_2-4$ generators of degree -1 and no relations in degree one. Since $L_1 \simeq L_1^*$, by minimality $L_1= \aG. S$ and we have: $\cdots \to \aG. S \to a.S\oplus (c_2-4).S(1)\oplus S(2) \to M \to 0$. By minimality $a=0$ and the conclusion follows.\\
(ii) By Lemma \ref{L-spe et gene H^1}, $M$ has at most $c_2-3$ generators of degree -1, so $\mu$ has rank $\geq 3$. If the rank is three there is one relation of degree one and we have: $$\cdots \to S(-1)\oplus \aG. S\oplus S(1) \to a.S\oplus (c_2-3).S(1)\oplus S(2) \to M \to 0$$ 
The induced minimal monad is:
$$\oc (-2)\oplus (c_2-3).\oc (-1) \oplus a.\oc \hookrightarrow \oc (-1)\oplus \aG. \oc \oplus \oc (1)$$
Since this is a minimal injective morphism of vector bundles we get $a=0$ and the conclusion follows.
\end{proof}

\begin{lemma}
\label{c1=0 sp -1 no}
Let $E$ be a stable rank two vector bundle with $c_1=0, b(E)=3$ and $h^1(E(-2))\neq 0$. Then the spectrum of $E$ is $Sp(E) = \{-1^2, 0^{c_2-4}, 1^2\}$.
\end{lemma}

\begin{proof} We have to show that $Sp(E)=\{-1,0^{c_2-2}, 1\}$ is impossible. By Lemma \ref{L-minMonad c1=0 -1} and Lemma \ref{L-bd c2 c1=0}, $c_2\leq 5$.  

If we are in case (ii) of Lemma \ref{L-minMonad c1=0 -1}  there exists a \emph{special} plane $H_0$ such that $H^1(E(-2)) \stackrel{m_{H_0}}{\to} H^1(E)$ is zero. It follows that $h^0(E_{H_0}(-1))\neq 0$. Since $h^0(E_H(-2))=0, \forall H$ (because $h^1(E(-3))=0$), the section of $E_{H_0}(-1))$ vanishes in codimension two: $0 \to \oc _{H_0} \to E_{H_0}(-1) \to \ic _Z(-2) \to 0$. We have $\deg (Z)=c_2+1$. Since $h^1(E(2))=0$ (because $b(E)=3$ and $H^1_*(E)$ is generated in degrees $\leq -1$), we get: $h^1(E_{H_0}(2))=0$ (because $h^2(E(1))=0$). It follows that $h^1(\ic _Z(1))=0$, which is absurd since $c_2 \geq 3$.

So we are necessarily in case (i) of Lemma \ref{L-minMonad c1=0 -1}, hence $c_2 \geq 4$. If $c_2=4$, then $H^1_*(E)=(S/I)(2)$, where $I$ is a complete intersection of type $(2,2,2,2)$. It follows that $d(E)=b(E)= 5$.

Assume $c_2=5$. The map $\left( S^2V \otimes \lag \xi \rag\right) \oplus \left( V \otimes \lag \aG \rag\right) \to H^1(E)$ is surjective. Since $h^1(E)=8$, we deduce that the map $S^2V \otimes \lag \xi \rag \to H^1(E)$ has an image, $W$, of dimension $\geq 4$. Since $b(E)=3$, if $H$ is any plane $H^1(E) \stackrel{m_H}{\to}H^1(E(1))$ has $W$ in its kernel, $K_H$. Since $h^0(E_H(1)) \geq \dim (K_H)$, this contradicts Theorem \ref{T-res}.
\end{proof}

Now we turn to the case $Sp(E) =\{-1^2, 0^{c_2-4}, 1^2\}$ (observe that necessarily $c_2 \geq 5$).

\begin{lemma}
\label{L-reso -1^2}
Let $E$ be a stable rank two vector bundle with $c_1=0$ and $Sp(E)=\{-1^2,0^{c-4},1^1\}$ ($c:=c_2 \geq 5$). Then the minimal free resolution of $H^1_*(E)$ is:
$$\cdots \to (8-e).S(-1)\oplus (2c-10).S\oplus (8-e).S(1) \to (c-e).S(1) \oplus 2.S(2) \to H^1_*(E) \to 0$$
where $5 \leq e \leq 8$. In particular $H^1_*(E)$ is generated in degrees $\leq -1$.
\end{lemma}

\begin{proof} Since $\rG (-1) \leq s(0)-1=c-5$, the image of $H^1(E(-2))\otimes V \to H^1(E(-1))$ has dimension $e \geq 5$. There are $c-e$ generators of degree $-1$ and exactly $8-e$ linear relations between the two generators of degree $-2$. Hence the resolution has the following shape:
$$\cdots \to \bigoplus S(b_i) \oplus (8-e).S(1) \to a.S \oplus (c-e).S(1)\oplus 2.S(2) \to H^1_*(E) \to 0$$
Since $L_1 = \bigoplus S(b_i) \oplus (8-e).S(1)$ satisfies $L_1 \simeq L_1^*$, we get $L_1= (8-e).S(-1)\oplus \aG .S\oplus (8-e).S(1)$. Now the minimal monad provides a minimal injective morphism of vector bundles: $\lc _0^* \hookrightarrow \lc _1$. It follows (by minimality) that $a.\oc \hookrightarrow (8-e).\oc (1)$. The quotient is a vector bundle with $H^1_*=0$ so it has to have rank $\geq 3$. This implies $8-e \geq a+3$. Since $e \geq 5$ it follows that $a=0$: there is no generator of degree zero. So the resolution is:
$$\cdots \to (8-e).S(-1)\oplus \aG .S\oplus (8-e).S(1) \to (c-e).S(1)\oplus 2.S(2) \to H^1_*(E) \to 0$$
Since $2.rk(L_0)+2 = rk (L_1)$, we get the result. 
\end{proof}

We are close to the end:

\begin{corollary}
\label{C-c1=0 cUnInstanton}
Let $E$ be a stable rank two vector bundle with $c_1=0$. If $b(E)=3$, then $E$ is an instanton with $3 \leq c_2(E) \leq 5$.
\end{corollary}

\begin{proof} If $h^1(E(-2))\neq 0$, by Lemma \ref{c1=0 sp -1 no} the spectrum is $\{-1^2, 0^{c_2-4},1^2\}$. According to Lemma \ref{L-reso -1^2} if $E$ has such a spectrum, $H^1_*(E)$ is generated in degrees $\leq -1$. By Lemma \ref{L-bd c2 c1=0}, $c_2\leq 5$. So it remains to show that the case $Sp(E)=\{-1^2,0,1^2\}$ is impossible. 
By Lemma \ref{L-spe et gene H^1}, $H^1_*(E)$ is generated by its degree -2 piece. Hence $d(E)=b(E)$ (Proposition \ref{P-generalities b(E)}). If $b(E)=3$, then $h^1(E(1))=0$ (Proposition \ref{P-generalities b(E)}). Since $\chi (E(1))= -7$, this is impossible.
\end{proof}
\medskip

\subsection{Instanton bundles with $b=3$.}\quad \\

We recall that an \emph{instanton} is a stable rank two vector bundle, $E$, on $\Pt$ with $c_1(E)=0$ and $h^1(E(-2))=0$. Equivalently $E$ is an instanton if it is stable and its spectrum is $\{0^{c_2}\}$. As it is well known $H^1_*(E)$ is generated by its degree -1 piece, hence (Proposition \ref{P-generalities b(E)}) $d(E)=b(E)$.

We recall an important result, due to Hartshorne-Hirschowitz (\cite{HaHi}):

\begin{theorem}
\label{T-InstNatCoh}
For every $c_2\geq 1$ there exists an instanton bundle with Chern classes $c_1=0, c_2$ and with natural cohomology (i.e. at most one of the four groups $H^i(E(k)), 0 \leq i \leq 3$ is non-zero, $\forall k \in \bZ$).
\end{theorem}

\begin{corollary}
\label{C-Instb=3}
There exists an instanton bundle, $E$, with $b(E)=3$ if and only if $c_2(E) \in \{3,4,5\}$.
\end{corollary}

\begin{proof} Since $\chi (E(2))=20-4c_2$ is $<0$ if $c_2\geq 6$, we have $h^1(E(2))\neq 0$, hence $d(E)\geq 4$. Since $b(E)=d(E)$ for an instanton, we conclude that if $b(E)=3$, then $c_2(E) \leq 5$.

If $E$ has natural cohomology $h^1(E(2))=0 \Leftrightarrow c_2 \leq 5$. Moreover since $\chi (E(1))=8-3c_2$, $h^1(E(1))\neq 0$ if $c_2\geq 3$. In conclusion, if $E$ has natural cohomology: $d(E)=3 \Leftrightarrow 3\leq c_2\leq 5$. Since in any case $d(E)=b(E)$ for an instanton, we conclude.
\end{proof}

Gathering everything together:

\begin{proposition}
\label{P-c1=0 final}
Let $E$ be a stable rank two vector bundle with $c_1=0$. If $b(E)=3$, then $E$ is an instanton with $3 \leq c_2(E) \leq 5$. Moreover for any $3\leq c_2 \leq 5$ there exists an instanton, $E$, with $c_2(E)=c_2$ and $b(E)=3$.
\end{proposition}

\newpage

\section{Negative instanton bundles.}
\label{sec-neg inst}

Let us start with a definition:

\begin{definition}
\label{D-neg instanton}
A \emph{negative instanton} is a stable rank two vector bundle, $E$, with $c_1(E)=-1$ and $h^1(E(-2))=0$.
\end{definition}

Equivalently $E$ is a negative instanton if it is stable with spectrum $\{-1^{c_2/2},0^{c_2/2}\}$. Although there are some analogies with the case $c_1=0$, the situation is quite different. For instance if $E$ is a negative instanton then $H^1_*(E)$ is not necessarily generated by its elements of degree -1. All we can say is that $H^1_*(E)$ is generated in degrees $\leq 0$ (Lemma \ref{L-spe et gene H^1}). We denote by $c$ the number of minimal generators of degree zero. Also we set $n:= c_2/2$. To conclude the proof of Theorem \ref{T-the thm} we will need in the next section the following:

\begin{proposition}
\label{P-bd c n<4}
Let $E$ be a negative instanton with $4 \leq c_2 \leq 6$, then $c \leq c_2/2 -1$. Moreover if $c_2=4$ and $c=1$, then $h^0(E(1))\neq 0$.
\end{proposition}

This is a particular case of the following result proved in \cite{EG-2}:

\begin{theorem}
\label{T-c<n}
Let $E$ be a negative instanton with $c-2 \geq 2$. Then $c \leq c_2/2 -1$. Moreover if $c=c_2/2-1$, then $h^0(E(1)) =1$. Finally for every $c_2 \geq 2$ there exists a negative instanton with $c = c_2/2 -1$.
\end{theorem}

However to keep this paper self-contained we will proceed now to prove Proposition \ref{P-bd c n<4} with an ad-hoc argument (completely different from the one used in \cite{EG-2}), see Corollary \ref{C-pf Prop}.

Notice by the way that it is easy to get the bound $c \leq n$: let $L$ be a general line. By combining $0 \to \ic _L \to \oc \to \oc _L \to 0$, and
$0 \to \oc (-2) \to 2.\oc (-1) \to \ic _L \to 0$, twisted by $E$, we get $2.H^1(E(-1)) \stackrel{j}{\to} H^1(E \otimes \ic _L) \stackrel{p}{\to}H^1(E)$. Now $j$ is injective and $p$ is surjective with $Ker(p) = H^0(E_L)$. We conclude with Riemann-Roch.


\subsection{Negative instantons with small Chern classes.}\quad \\

let $E$ be a negative instanton, we set $n:=c_2/2$ and denote by $c$ the number of minimal generators of $H^1_*(E)$ of degree zero. We assume $c>0$. We know that $c\leq n$. The minimal monad is:

\begin{equation}
\label{eq: min monad}
n.\oc (-2)\oplus c.\oc (-1) \hookrightarrow (c+n+1).\oc (-1)\oplus (c+n+1).\oc \twoheadrightarrow c.\oc\oplus n.\oc (1)
\end{equation}

The display of the monad is:

$$\begin{array}{ccccccccc}
 & & & &0& &0 & & \\
 & & & &\downarrow& &\downarrow & & \\
0 & \to &c.\oc (-1)\oplus n.\oc (-2)& \to &\nc & \to &E&\to &0\\
 & & ||& &\downarrow& &\downarrow & & \\
0 & \to &c.\oc (-1)\oplus n.\oc (-2)& \stackrel{\bG}{\to} &(c+n+1).(\oc (-1)\oplus \oc ) & \to &\fc&\to &0\\
 & & & &\downarrow \aG& &\downarrow & & \\
 & & & &c.\oc \oplus n.\oc (1)& =&c.\oc \oplus n.\oc (1) & & \\
 & & & &\downarrow& &\downarrow & & \\
 & & & &0& &0 & & \end{array}$$
 
\noindent By minimality $\bG$ induces: 
\begin{equation}
\label{eq:bG tilde}
c.\oc (-1) \stackrel{\tilde \bG}{\hookrightarrow} (c+n+1).\oc
\end{equation} 
Also $\aG$ induces: 
\begin{equation}
\label{eq.aG tilde}
(c+n+1).\oc (-1) \stackrel{\tilde \aG}{\twoheadrightarrow} c.\oc
\end{equation}

The first main remark is:

\begin{lemma}
\label{L-cok bG tilde loc free}
With notations as above, $\ec := Coker\, (\tilde \bG )$ is locally free.
\end{lemma}

\begin{proof} By dualizing the display of the monad and since $E^*(-1) \simeq E$ we see that (up to isomorphism) $\aG ^*(-1) = \bG$ and also $\tilde \aG ^*(-1)=\tilde \bG$. Now we have an exact sequence:
$$0 \to \ac \to (c+n+1).\oc (-1) \stackrel{\tilde \aG}{\to} c.\oc \to 0$$
where $\ac$ is a vector bundle. By the above remark:
$$0 \to c.\oc (-1) \stackrel{\tilde \bG}{\to} (c+n+1).\oc \to \ec \simeq \ac ^* \to 0$$
\end{proof} 

The map $\tilde \aG$ yields the following commutative diagram:

$$\begin{array}{ccccccccc}
 & &0 & &0& &0 & & \\
 & &\downarrow & &\downarrow& &\downarrow & & \\
0& \to &K&\to& (c+n+1).\oc &\stackrel{\psi}{\to}& n.\oc(1)& & \\
 & &\downarrow & &\downarrow& &\downarrow & & \\
0& \to &\nc&\to& (c+n+1).(\oc(-1)\oplus \oc ) &\stackrel{\aG}{\to}& c.\oc \oplus n.\oc(1)&\to &0 \\
 & &\downarrow \lG & &\downarrow& &\downarrow & & \\
0& \to &\ac&\to& (c+n+1).\oc(-1) &\stackrel{\tilde \aG}{\to}& c.\oc&\to &0 \\
 & & & &\downarrow& &\downarrow & & \\
  & & & &0& &0 & & \end{array}$$

The map $\psi$ need not be surjective. The snake lemma applied to the two bottom row of the diagram shows that: $Coker(\lG )\simeq Coker(\psi )$. Let us define $J := Im(\psi )$.

\begin{lemma}
\label{L-h0 K,N,E}
With notations as above:\\
(i) $h^0(K)=h^0(\nc )=h^0(E)=0$\\
(ii) $h^0(K(1)) \leq h^0(\nc (1)) = c+h^0(E(1)) \leq c+1$\\
(iii) $h^0(K(1)) \geq c$.
\end{lemma}

\begin{proof} The first two statements follow easily from the display of the monad and the diagram above (taking into account that $h^0(E(1)) \leq 1$ by Theorem \ref{T-res}).

For (iii) consider the following diagram:
$$\xymatrix{
c.\oc (-1) \ar@{^{(}->}[rrd]^i& &K \ar@{^{(}->}[d]\\
 & &\nc \ar@{^{(}->}[r]^>>>>j \ar [d]^p &(c+n+1).(\oc (-1)\oplus \oc) \ar@{->>}[d]^ \pi \\
 & &\ac \ar@{^{(}->}[r]^>>>>>>>s & (c+n+1).\oc (-1)}$$
We have $\pi \circ j\circ i=0$ by the monad. So $s\circ p \circ i=0$. Since $s$ is injective $p\circ i=0$ and $c.\oc (-1) \stackrel{i}{\hookrightarrow} \nc$ factors through $K$.
\end{proof}

\begin{corollary}
\label{C-rkJ=n} With notations as above, if $rk(J)=n$, then $c \leq n-1$. Moreover if $c=n-1$, then $h^0(E(1))=1$.
\end{corollary}

\begin{proof} By Lemma \ref{L-h0 K,N,E} we have a commutative diagram:
$$\xymatrix{
0\ar [r] &c.\oc (-1)\ar [r] \ar@{=}[d] &K\ar [r] \ar@{^{(}->}[d]& \jc \ar [r]\ar@{^{(}->}[d]& 0\\
0\ar [r] &c.\oc (-1)\ar [r] &(c+n+1).\oc \ar [r]\ar@{->>}[d]&\ec \ar [r]\ar@{->>}[d]& 0\\
 & &J \ar@{=}[r]& J}$$
Since $rk(J)=n$, we get $rk(\jc )=1$. By Lemma \ref{L-cok bG tilde loc free} $\ec$ is locally free. Since $J$ is torsion free, $\jc$ is reflexive. So $\jc = \oc (a)$ and $K = c.\oc (-1)\oplus \oc (a)$. From $h^0(K)=0$ (Lemma \ref{L-h0 K,N,E}), we get $a<0$. Now $c_1(J)=-c_1(K)= c-a$. Since $c_1(J) \leq n$ (because $J \subset n.\oc (1)$), we have $c =c_1(J)+a < n$. Finally if $c=n-1$, the only possibility is $c_1(J)=n$, $a=-1$. So $h^0(K(1))=c+1$ and we conclude with Lemma \ref{L-h0 K,N,E}.
\end{proof} 

Now we have the following simple lemma:

\begin{lemma}
\label{L-F dans n.O Pk}
Let $\fc$ be a coherent sheaf of rank $r$ on $\Pk$, $k \geq 1$, such that $\fc \subset n.\oc _{\Pk}$. Then $h^0(\fc (m)) \leq r.h^0(\oc _{\Pk}(m))$, for every $m \in \bZ$. Moreover if there is equality for some $m \geq 0$, then $\fc = r.\oc _{\Pk}$.
\end{lemma}

\begin{proof} We make a double induction on $k,m$. If $k=1$, $\fc = \bigoplus _{i=1}^r \oc _{\Pun}(a_i)$ with $a_i \leq 0$ and the statement follows immediately. Assume the Lemma proved for $k-1$. Since $\fc \subset n.\oc _{\Pk}$, $h^0(\fc (-1))=0$. Let $H$ be a general hyperplane. We have $\fc _H \subset n.\oc _H$ and an exact sequence $0\to \fc (m-1) \to \fc (m) \to \fc _H(m)\to 0$. We get $h^0(\fc ) \leq h^0(\fc _H)\leq r$. Then we conclude by induction on $m, m\geq 0$.

If $h^0(\fc (m))=r.h^0(\oc _{\Pk}(m))$ for some $m \geq 0$, then by descending induction $h^0(\fc )=r$. The evaluation map yields $0 \to r.\oc \to \fc \to \gc \to 0$. The inclusion $r.\oc \hookrightarrow n.\oc$ shows that $\gc \hookrightarrow (n-r).\oc$. Since $\gc$ has rank zero, it follows that $\gc =0$. 
\end{proof}

By considering $\fc = r.\oc _{\Pk}$ we see that the Lemma is sharp.

Now we turn back to $\Pt$ and the application we had in mind, i.e. the proof of Proposition \ref{P-bd c n<4}:

\begin{corollary}
\label{C-pf Prop}
(1) Let $0 \to K \to (n+c+1).\oc \to J \to 0$, be an exact sequence with $J \subset n.\oc (1)$. Assume $h^0(K(1)) \leq c+1$, with $c\leq n$. If $n \geq 2$, then $rk(J) \geq 2$. Moreover if $2 \leq n \leq 4$, then $c \leq n-1$ or $rk(J)=n$.\\
(2) Let $E$ be a negative instanton with $4\leq c_2 \leq 8$, then $c \leq c_2/2 -1$, where $c$ is the number of minimal generators of degree zero of $H^1_*(E)$. Moreover if $c_2=4$ and $c=1$, then $h^0(E(1))=1$.
\end{corollary}

\begin{proof} (1) We have $h^0((n+c+1).\oc (1)) \leq h^0K(1))+h^0(J(1))$, hence $4(n+c+1) \leq c+1+h^0(J(1))$. By Lemma \ref{L-F dans n.O Pk}: $h^0(J(1)) \leq 10r$, where $r:=rk(J)$. It follows that $4n+3c+3 \leq 10r$. Hence $r \geq 2$ if $n \geq 2$.

If $c=n$ we get $7n+3 \leq 10r$. If $r \leq n-1$, then $13\leq 3n$, hence $n \geq 5$.\\
(2) Follows from (1) above and Corollary \ref{C-rkJ=n}. 
\end{proof}

\section{Stable bundles with $c_1=-1$ and $b=3$.}

In this section $E$ will denote a stable rank two vector bundle on $\Pt$ with Chern classes $(-1,c_2)$ and with $b(E)=3$. For such a bundle we have:

\begin{lemma}
\label{L-c1=-1 first}
With notations as above $h^1(E(-2))=0$, $h^1(E(3))=0$ and $c_2 \leq 6$.
\end{lemma}

\begin{proof} By Proposition \ref{P-generalities b(E)} (iv), $h^1(E(-2))=0$. In particular (Lemma \ref{L-spe et gene H^1}) $H^1_*(E)$ is generated in degrees $\leq 0$. By Proposition \ref{P-generalities b(E)} (ii), we get $h^1(E(3))=0$. Since $h^3(E(3))=0$ it follows that $\chi (E(3)) \geq 0$. Since $\chi (E(3))= 30-(9c_2)/2$, we get $c_2 \leq 6$.
\end{proof}

Since $c_2$ is even we are left with three cases: $c_2\in \{2, 4, 6\}$.

\begin{lemma}
\label{c1=-1 c2=2 stable}
Every stable rank two vector bundle, $E$, with $c_1=-1, c_2=2$ has $b(E)=d(E)=3$.
\end{lemma}

\begin{proof} We have (\cite{Ha-Sols} Prop. 2.2) that $H^1_*(E)$ is concentrated in degrees -1, 0, 1 with $h^1(E(-1))=h^1(E(1))=1$, $h^1(E)=2$. The module $H^1_*(E)$ is isomorphic (up to twist) to the Hartshorne-Rao module of the disjoint union of two conics, such a module is generated by its lowest degree piece.
\end{proof}

Concerning the case $c_2=4$ we first recall (see \cite{Banica-Manolache}):

\begin{lemma}
\label{L-BanicaMano}
Let $E$ be a stable rank two vector bundle with $c_1=-1, c_2=4$, then $h^1(E(2)) = h^0(E(2)) \neq 0$.
\end{lemma}

\begin{proposition}
Let $E$ be a stable rank two vector bundle with $c_1=-1, c_2=4$. Then $b(E)\geq 4$.
\end{proposition}

\begin{proof}From Lemma \ref{L-BanicaMano} it turns out that $d(E) \geq 4$. The module $H^1_*(E)$ is generated in degrees -1, 0. If there is no generator in degree 0 then by Proposition \ref{P-generalities b(E)}, $b(E)=d(E) > 3$. So we may assume that $H^1_*(E)$ has some generator of degree zero, i.e. (Corollary \ref{C-pf Prop}) one generator of degree zero. So the image, $W$, of $\mu : H^1(E(-1))\otimes V \to H^1(E)$ has dimension 4. Furthermore, always by Corollary \ref{C-pf Prop}, $h^0(E(1))=1$. If $H$ is a general plane we have $0 \to \oc _H \to E_H(1) \to \ic _Z(1) \to 0$, where $\deg (Z)=4$ and (Theorem \ref{T-res}) $h^0(\ic _Z(1))=0$. It follows that $h^1(\ic _Z(2))=h^1(E_H(2))=0$. Moreover the exact sequence $0 \to E \to E(1) \to E_H(1) \to 0$ yields an inclusion $H^1(E) \hookrightarrow H^1(E(1))$. Let $W' \subset H^1(E(1))$ be the image of $W$. The assumption $b(E)=3$ implies that $W'$ is contained in the kernel of $H^1(E(1)) \stackrel{.H}{\to} H^1(E(2))$. It follows that $h^0(E_H(2))\geq 4$. In conclusion we have: $\cdots \to H^0(E_H(2)) \to H^1(E(1)) \stackrel{\fG}{\to} H^1(E(2)) \to 0$ and $W' \subset Ker(\fG )$. By Riemmann-Roch $h^1(E(1))=6$. Since $h^0(E(1))=1$, we get $h^0(E(2))\geq 4$, hence (Lemma \ref{L-BanicaMano}), $h^1(E(2))\geq 4$. It follows that $\dim (Ker(\fG ))\leq 2$. This is a contradiction since $\dim W' =4$.
\end{proof}

Finally let's turn to the last case $c_2=6$. First we have:

\begin{lemma}
\label{L-coh ordre jline}
let $E$ be a stable rank two vector bundle with $c_1=-1$. Assume $h^1(E(3))=0$ and $h^1(E(-2))=0$. If $L$ is a line and if $E_L \simeq \oc _L(a)\oplus \oc _L(-a-1), a \geq 0$, then $a < 4$.
\end{lemma}

\begin{proof} Assume $a \geq 4$ for some line $L$. The exact sequence $0 \to \ic _L\otimes E \to E \to E_L \to 0$ shows that $h^1(\ic _L\otimes E(-a))\neq 0$. Now consider the exact sequence:
$0 \to E(-2) \to 2.E(-1) \to \ic _L\otimes E \to 0$. Since $h^1(E(-1-a))=0$, from $h^1(\ic L\otimes E(-a))\neq 0$, we get $h^2(E(-a-2))\neq 0$. By duality $h^2(E(-a-2))=h^1(E(a-1))$. Since $a-1\geq 3$, this is impossible ($h^1(E(3))=0$ and Theorem \ref{T-Buraggina}).
\end{proof}

We can now complete the proof of Theorem \ref{T-the thm}:

\begin{proposition}
Let $E$ be a stable rank two vector bundle with $c_1=-1, c_2=6$. Then $b(E)>3$.
\end{proposition}

\begin{proof} First observe that, by Corollary \ref{C-pf Prop}, the natural map $\mu :H^1(E(-1))\otimes V \to H^1(E)$ has an image, $W$, of dimension $\geq 6$.

Assume $h^0(E_H(1))=0$ for $H$ a general plane. Then we have $H^1(E) \hookrightarrow H^1(E(1))$. Let $W'$ denote the image of $W$. Now twisting by one we have $\cdots \to H^0(E_H(2)) \to H^1(E(1)) \stackrel{\fG}{\to} H^1(E(2))$. If $b(E)=3$, then $W' \subset Ker(\fG )$ and this implies $h^0(E_H(2)) \geq 6$. Since $h^0(E_H(1))=0$, $E_H(2)$ has a section vanishing in codimension two: $0 \to \oc _H \to E_H(2) \to \ic _Z(3) \to 0\,\,(+)$, where $\deg (Z)=8$. Since $h^0(\ic _Z(3))\geq 5$ and $h^0(\ic _Z(2))=0$, $Z$ has seven points lying on a line $L$. Restricting the exact sequence to $L$ we get $E_L(2) \twoheadrightarrow \oc _L(-4)$. It follows that $E_L \simeq \oc _L(5)\oplus \oc _L(-6)$. By Lemma \ref{L-coh ordre jline} this is impossible. hence $b > 3$ if $h^0(E_H(1))=0$.

Assume $h^0(E_H(1))\neq 0$. Then we have $0 \to \oc _H \to E_H(1) \to \ic _Z(1) \to 0\,\,(*)$, where $\deg (Z)=6$. By Theorem \ref{T-res}, $h^0(\ic _Z(1))=0$. With notations as above $W' \subset H^1(E(1))$ has $\dim W' \geq 5$, hence $h^0(E_H(2)) \geq 5$. This implies $h^0(\ic _Z(2)) \geq 2$. Since $h^0(\ic _Z(1))=0$, it follows that $Z$ has 5 points on a line $R$. Restricting $(*)$ to $R$ we get: $E_R(1) \twoheadrightarrow \oc _R(-4)$. it follows that $E_R \simeq \oc _R(4)\oplus \oc _R(-5)$; in contradiction with Lemma \ref{L-coh ordre jline}. So $b>3$ again.
\end{proof}



\end{document}